\documentclass{amsart}
\usepackage{amssymb, latexsym}
\usepackage{graphicx}
\graphicspath{{images/}}
\usepackage[font=small, labelfont=bf]{caption}

\newtheorem{theorem}{Theorem}
\newtheorem{lemma}{Lemma}
\newtheorem{remark}{Remark}

\newtheorem{example}{Example}

\begin{document}
\title{Conjugate Words and Intersections of Geodesics in $\mathbb{H}^2$}
\author{Rita Gitik}
\address{ Department of Mathematics \\ University of Michigan \\ Ann Arbor, MI, 48109}
\email{ritagtk@umich.edu}
\date{\today}

\begin{abstract}
We investigate intersections of geodesic lines in $\mathbb{H}^2$ and in an associated tree $T$, proving the following result.
Let $M$ be a punctured hyperbolic torus and let $\gamma$ be a closed geodesic in $M$.
Any edge of any triangle formed by distinct geodesic lines in  the preimage of $\gamma$ in $\mathbb{H}^2$ is shorter then $\gamma$.
However, a similar result does not hold in the tree T.  
Let $W$ be a reduced and cyclically reduced word in $\pi_1(M)= \langle x, y \rangle$. We construct several examples of triangles in $T$ formed by distinct axes in $T$ stabilized by conjugates of $W$ such that an edge in those triangles is longer than $L(W)$. 
We also prove that if $W$  overlaps two of its conjugates in such a way that the overlaps cover all of $W$ and the overlaps do not intersect, then there exists a decomposition $W=BC^kI, k > 0$, with $B$ a terminal subword of $C$ and $I$  an initial subword of $C$.
\end{abstract}

\subjclass[2010]{Primary: 53C22; Secondary: 30F45, 20E05, 20E45}

\maketitle

\section{Introduction}

The study of curves on surfaces is a classical subject going back to the origins of topology, \cite{Bi}. Of particular interest are closed geodesics which can be investigated by looking at their lifts in covering spaces of the surface, \cite{F-H-S}, \cite{Ga}, \cite{Gr}, \cite{H-S2}, \cite{H-S3}, \cite{NC}. In this paper we consider hyperbolic surfaces and study the intersections of geodesic lines in $\mathbb{H}^2$, \cite{H-S1}. In general, the patterns of such intersections are very complicated, so we restrict ourselves to three geodesic lines in $\mathbb{H}^2$ which are lifts of the same closed geodesic in a punctured hyperbolic surface. For the sake of clarity we choose the surface to be a punctured torus.

An important tool in studying geodesic lines in $\mathbb{H}^2$ is  the tree $T$ in $\mathbb{H}^2$, defined as follows, cf. \cite{H-S1}, pp.111-112.

Let $M$ be a hyperbolic punctured torus and let $x_0$ and $y_0$ be disjoint infinite geodesic arcs on $M$ such that $M$ cut along $x_0 \cup y_0$ is an open two-dimensional disk $D$. There exist closed geodesics  $x$ and $y$ in $M$  such that $x \cap x_0=$point, $x \cap y_0 =\varnothing, y \cap x_0=\varnothing$, and $y \cap y_0=$point, which generate the fundamental group of $M$. Note that the fundamental group of $M$  is a free group of rank two, 
$\pi_1(M)= \langle x, y \rangle$. The universal cover of $M$ is the hyperbolic plane $\mathbb{H}^2$, so $M= \pi_1(M) \diagdown \mathbb{H}^2$. Let $\widetilde{D}$ be a lift of the disc $D$ to $\mathbb{H}^2$. Note that $\mathbb{H}^2$ is tiled by the translates of the closure of $\widetilde{D}$ by $\pi_1(M)$. Let $T$ be the graph in $\mathbb{H}^2$ dual to this tiling, i.e. the vertices of $T$ are located one in each translate of $\widetilde{D}$, and each edge of $T$ connects two vertices of $T$ in adjacent copies of
$\widetilde{D}$, so each edge intersects  one lift of either $x_0$ or $y_0$ once. As $\mathbb{H}^2$ is simply connected, $T$ is a tree. Note that $T$ is the Cayley graph of  $\pi_1(M)= \langle x, y \rangle$. Define the distance $d_T(v,u)$ between two vertices of $T$ to be the number of edges in a shortest path in $T$ connecting $v$ and $u$.

Any element $f$ of $\pi_1(M)$ acts on $T$ leaving invariant a unique line, called the axis of $f$, which contains all vertices $v$ with minimum
$d_T(v, f(v))$. That minimum is called the translation length of $f$, and is equal to the length of the word $W$ in $\pi_1(M)= \langle x, y \rangle$ obtained from $f$ by reduction and cyclic reduction. Denote the length of the word $W$ in $\pi_1(M)= \langle x, y \rangle$ by $L(W)$.

We prove the following result in Section 2.

\begin{theorem}
Let $M$ be a punctured hyperbolic torus and let $\gamma$ be a closed geodesic in $M$.
Any edge of any triangle formed by distinct geodesic lines in  the preimage of $\gamma$ in $\mathbb{H}^2$ is shorter then $\gamma$.
\end{theorem}

However, a similar result does not hold in the tree T. In Section 3 we construct several examples of triangles in $T$ formed by distinct axes in $T$ stabilized by conjugates of $W$, such that an edge in those triangles is longer than $L(W)$. 

In Section 4 we determine the general form of a reduced and cyclically reduced word $W$ in $\pi_1(M)= \langle x, y \rangle$ which  overlaps two of its conjugates in such a way that the overlaps cover all of $W$, proving the following result.

\begin{theorem}
Let $W$ be a reduced and cyclically reduced word in $\pi_1(M)= \langle x, y \rangle$ which overlaps two of its conjugates in such a way that the overlaps cover all of $W$ and the overlaps do not intersect. Then there exists a decomposition $W=BC^kI, k > 0$, with $B$ a terminal subword of $C$ and $I$  an initial subword of $C$.
\end{theorem}

\section{Triangles in $\mathbb{H}^2$}

We use the notation from the previous section.

\begin{lemma}
 Let $f$ be an element in $\pi_1(M)= \langle x, y \rangle$ and let $W$ be its reduced and cyclically reduced conjugate.
Consider two axes in the tree $T$ stabilized by $f$ and its conjugate
 $f'\in \pi_1(M)$. If such axes intersect in an interval labeled with a word $W_0$ such that $L(W_0)=L(W)-1$ then they coincide.
\end{lemma}

\begin{proof}
WLOG $W_0$ is an initial subword of $W$, hence WLOG there exists a decomposition $W=W_0x$, where $x$ is a generator of $\pi_1(M)= \langle x, y \rangle$. Let $W'$ be a reduced and cyclically reduced conjugate of $f'$ containing $W_0$. Then the abelianization of $W$  implies that either $W'=xW_0$ or $W'=W_0x=W$. In both cases the axes coincide.
\end{proof}

\textbf{Proof of Theorem 1}

Assume to the contrary that there exists a triangle in $\mathbb{H}^2$ formed by geodesic lines $l,m$, and $n$, which are distinct lifts of the geodesic $\gamma$, such that the length of the side lying in
$l$ is longer than $\gamma$. Note that $l$ is stabilized by some element $f$ in $\pi_1(F)$ which acts as a hyperbolic isometry of $\mathbb{H}^2$.
Let $P$ be the intersection of $l$ and $n$, and let $X$ be the intersection of $l$ and $m$. The length of $\gamma$ is equal to the length of the segment 
$Pf(P)$ which is equal to the length of the segment $f(P) f^2(P)$. 

Consider two cases.

\textbf{Case 1}. \emph{The side $PX$ of the triangle formed by lines $l,m$, and $n$ is shorter than the segment $P f^2(P)$.} 

See Figure 1.

\includegraphics[scale=0.5]{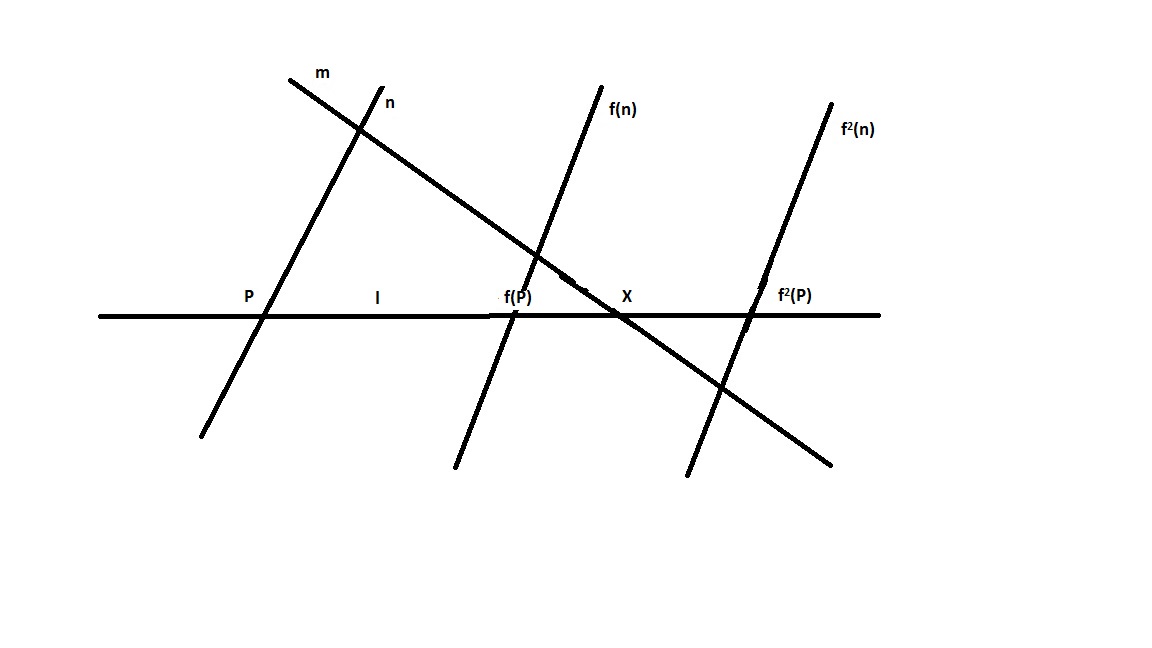}
\captionof{figure}{}

\bigskip

By assumption, the side $PX$  is longer than $\gamma$, so the segment $X f^2(P)$ is shorter than the segment $PX$.
Consider the geodesics $f(n)$ and $f^2(n)$.
As $f$ is an isometry, the geodesics $n, f(n)$, and $f^2(n)$ make the same angle with $l$. Then as $X f^2(P)$ is shorter than $PX$, the angle between $n$ and $l$ is equal to the angle between $f^2(n)$ and $l$, and the opposite angles between $m$ and $l$ are equal, it follows that $m$ and
 $f^2(n)$ intersect, as shown in Figure 1.

Let $T$ be the tree in $\mathbb{H}^2$ defined above and let $W$ be a reduced and cyclically reduced word conjugate to $f$ in $\pi_1(F)$. 
The geodesic lines $l, m$, and $n$  are transversal to the lifts of the geodesics $x_0$ and $y_0$ in $\mathbb{H}^2$.
Consider the intersections of the lifts of the geodesics $x_0$ and $y_0$ with lines $l,m$, and $n$.

Choose a projection $s: \mathbb{H}^2 \rightarrow T$ which respects the action of $\pi_1(F)$ on 
$\mathbb{H}^2$. It can be arranged that the restriction of $s$ to each component of the lift of $\gamma$ in $\mathbb{H}^2$ is monotone, so $s$ maps each component of the lift of $\gamma$ onto a geodesic in $T$.

Let $b$ lifts of $x_0$ and $y_0$ intersect both $l$ and $n$ to the left of the point $P$ and let $a$ lifts of $x_0$ and $y_0$ intersect both $l$ and $n$ to the right of the point $P$. Then there are $a+b$ lifts of $x_0$ and $y_0$ crossing $l$ and $n$, hence the length of the intersection $s(l) \cap s(n)$ is
 $a+b$. Lemma 1 implies that $a+b < L(W)-1$. By a similar argument, the number $c$ of the lifts of $x_0$ and $y_0$ intersecting both $l$ and $m$ is also less than $L(W)-1$. As $f$ is an isometry, there are $b$ lifts of $x_0$ and $y_0$ crossing $l$ and $f^2(n)$ to the left of $f^2(P)$. Then the total number of the lifts of $x_0$ and $y_0$ crossing $l$ between the points $P$ and $f^2(P)$ is $a+b+c$, which is strictly less than $2L(W)$. However by construction, the number of the lifts of $x_0$ and $y_0$ crossing  $l$ between the points $P$ and $f^2(P)$ should be equal to $2L(W)$.
   
This contradiction  completes the proof of Theorem 1 in Case 1.

\textbf{Case 2}. \emph{The side $PX$ of the triangle formed by lines $l,m$, and $n$ is longer or equal than the segment $P f^2(P)$.}

 See Figure 2.

\includegraphics[scale=0.5]{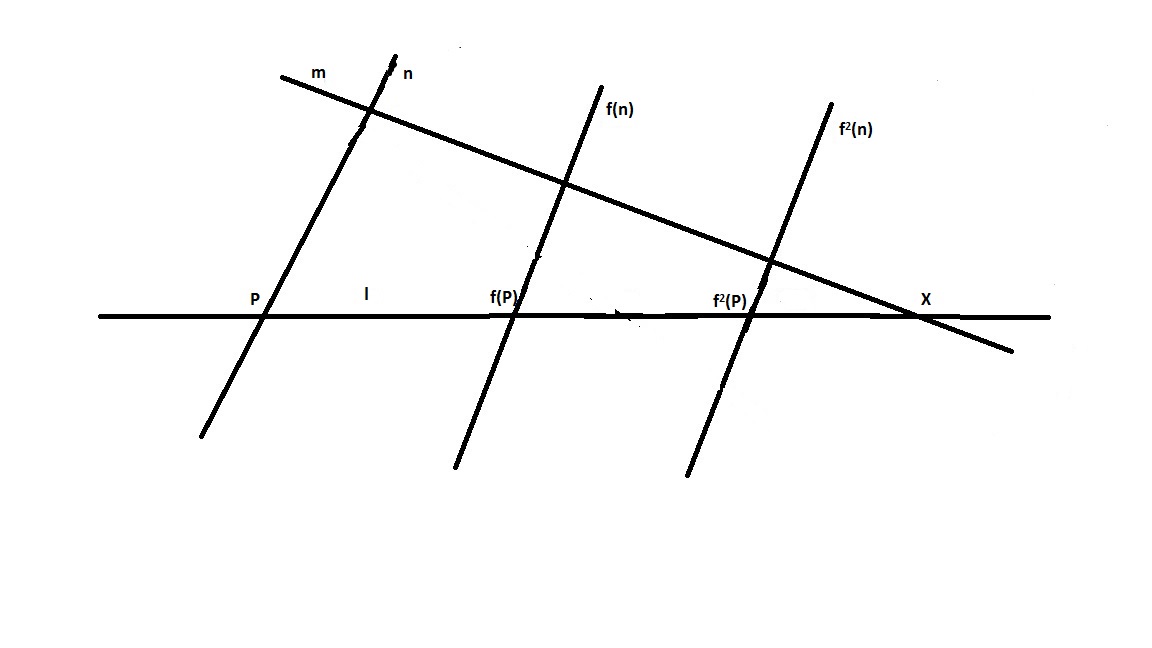}
\captionof{figure}{}

\bigskip

Let $a$ lifts of $x_0$ and $y_0$ intersect both $l$ and $n$ to the right of the point $P$. Then the length of the intersection $s(l) \cap s(n)$ is
 not shorter than $a$, hence Lemma 1 implies that $a < L(W)-1$. Let $c$ be the number of the lifts of $x_0$ and $y_0$ intersecting both $l$ and $m$ to the left of the point $f^2(P)$. Then the length of the intersection $s(l) \cap s(m)$ is not shorter than $c$, hence Lemma 1 implies that $c < L(W)-1$.
 Therefore the total number of the lifts of $x_0$ and $y_0$ crossing $l$ between the points $P$ and $f^2(P)$ is $a+c$, which is strictly less than 
 $2L(W)$. However by construction, the number of the lifts of $x_0$ and $y_0$ crossing  $l$ between the points $P$ and $f^2(P)$ should be equal to $2L(W)$.
   
This contradiction  completes the proof of Theorem 1 in Case 2. $\square$
\bigskip

The author would like to thank Max Neumann-Coto for sharing his ideas about Theorem 1.

\section{Triangles in the Tree $T$}

Consider again the tree $T$ defined above. As was mentioned already, $T$ can be considered to be the Cayley graph of the free group
$\pi_1(M)=<x,y>$. Let $W$ be a reduced and cyclically reduced word in $\{x, y, x^{-1}, y^{-1} \}$. Consider three distinct axes in $T$ stabilized by the word $W$ and two of its conjugates $f_1$ and $f_2$. Call the axes $\lambda, \lambda_1$, and $\lambda_2$. Let $\widetilde{W}$ denote the bi-infinite product of the word $W$. Note that all the axes $\lambda, \lambda_1$, and $\lambda_2$ are labeled by the bi-infinite word $\widetilde{W}$.

Choose a copy of the word $W$ in $\lambda$. We will work with that chosen copy. Assume that the axes intersect in such a way that $\lambda_1 \cap \lambda$ and $\lambda_2 \cap \lambda$ cover the word $W$ in $\lambda$. Note that the intervals $\lambda_1 \cap \lambda, \lambda_2 \cap \lambda$, and  $\lambda_1 \cap \lambda_2$ form a tripod in the tree $T$ which is a degenerate triangle. Denote the label of the interval $\lambda \cap \lambda_1$ by $U$ and the label of the interval $\lambda \cap \lambda_2$ by $V$. The four examples below show that, in contrast with Theorem 1, $W, \lambda_1$, and $\lambda_2$ can be chosen in such a way that $L(U \cup V) \geq L(W)$. 

\bigskip

Note that Lemma 1 implies that $L(U) \leq L(W)-2$ and $L(V)\leq L(W)-2$, so $L(U \cup V) \leq 2L(W)- 4$.

\bigskip

Let $\mu_i, i=1,2$ be a subinterval of $\lambda_i$ containing $\lambda_i \cap \lambda$ such that its label $W_i$ is a reduced and cyclically reduced conjugate of $f$. Then $W_1$ contains $U$ and  $W_2$   contains $V$.

\begin{example} See Figure 3.

Let $W=xyxyx, W_1=x^{-1}Wx=yxyx^2,$ and $W_2=xWx^{-1}=x^2yxy$. Then $U=W_1\cap W=xyx$ and $V=W_2 \cap W=xyx$, so that $U \cap V=x$.
Then $(W_1 \cup W_2) \cap W=U \cup V=W$.
\end{example}

\includegraphics[scale=0.6]{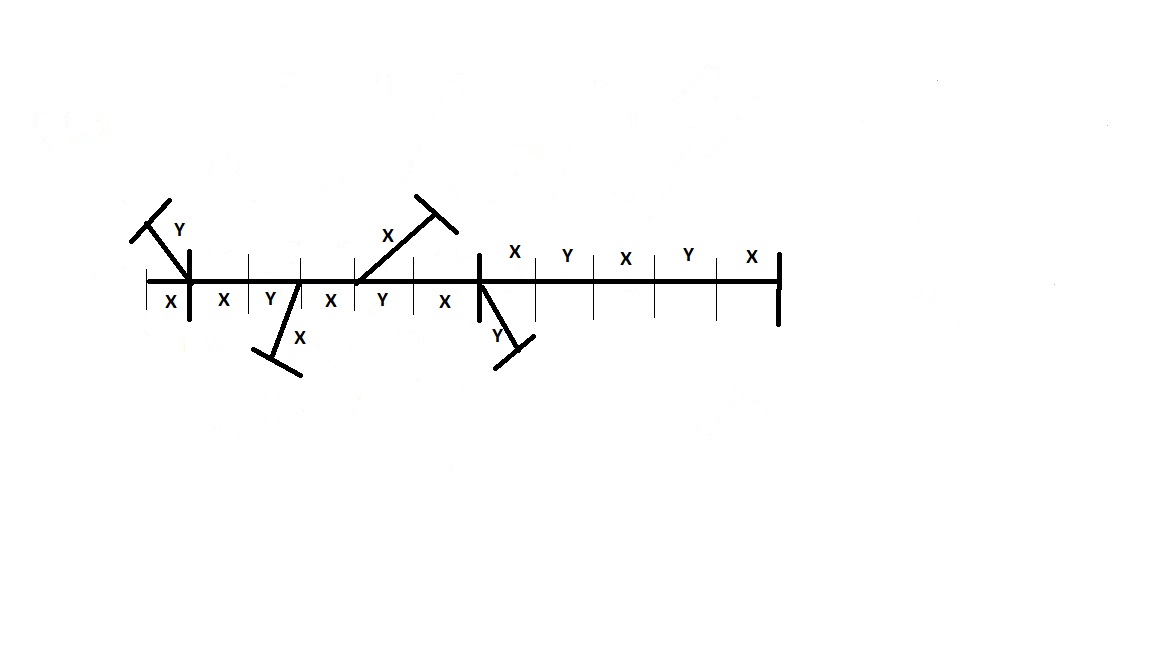}
\captionof{figure}{}

\bigskip

\begin{example} See Figure 4.

Let $W=xy^2xy^2x, W_1=y^{-1}x^{-1}Wxy=yxy^2x^2y,$ and $W_2=xWx^{-1}=x^2y^2xy^2$. Then $U=W_1\cap W=xy^2x$ and  $V=W_2 \cap W=xy^2x$, so that 
$U \cap V=x$. Then $(W_1 \cup W_2) \cap W=U \cup V=W$.
\end{example}

\includegraphics[scale=0.6]{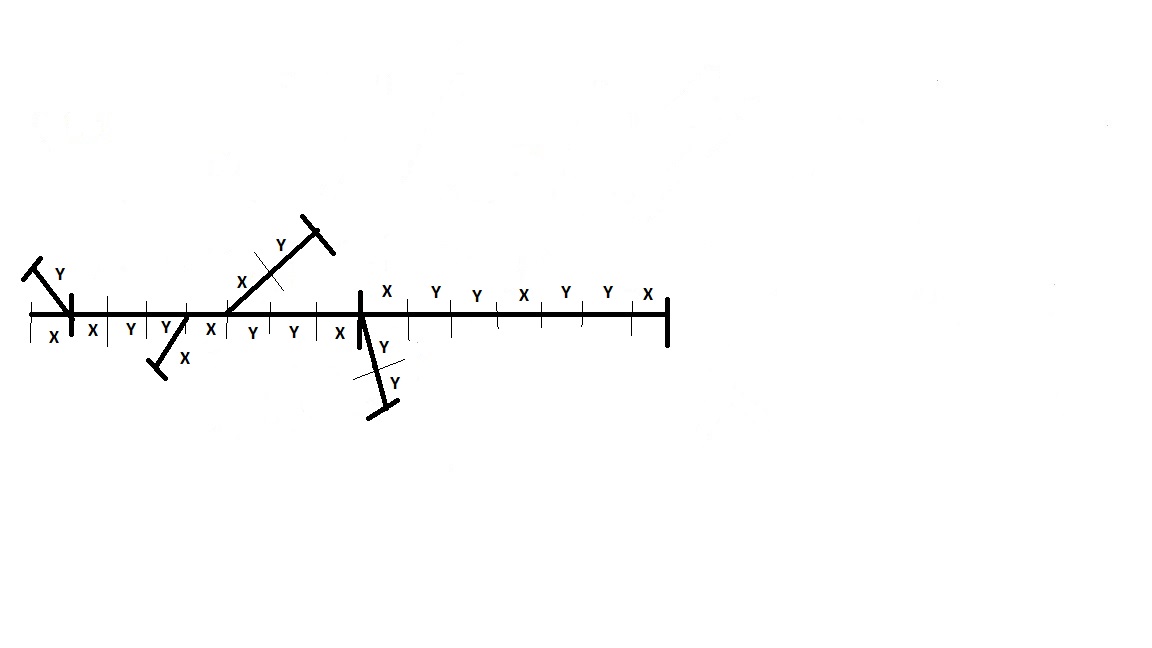}
\captionof{figure}{}

\begin{example} See Figure 5.

Let $W=yxy^2xy^2x, W_1=y^{-1}Wy=xy^2xy^2xy,$ and $W_2=y^2xWx^{-1}y^{-2}=y^2xyxy^2x$. Then $U=W_1\cap W=yxy^2xy$ and  $V=W_2 \cap \widetilde{W}=yxy$, so that $U \cap V=$point. Then $(W_1 \cup W_2) \cap \widetilde{W}=(U \cup V) \cap \widetilde{W}= Wy$.
\end{example}

\includegraphics[scale=0.6]{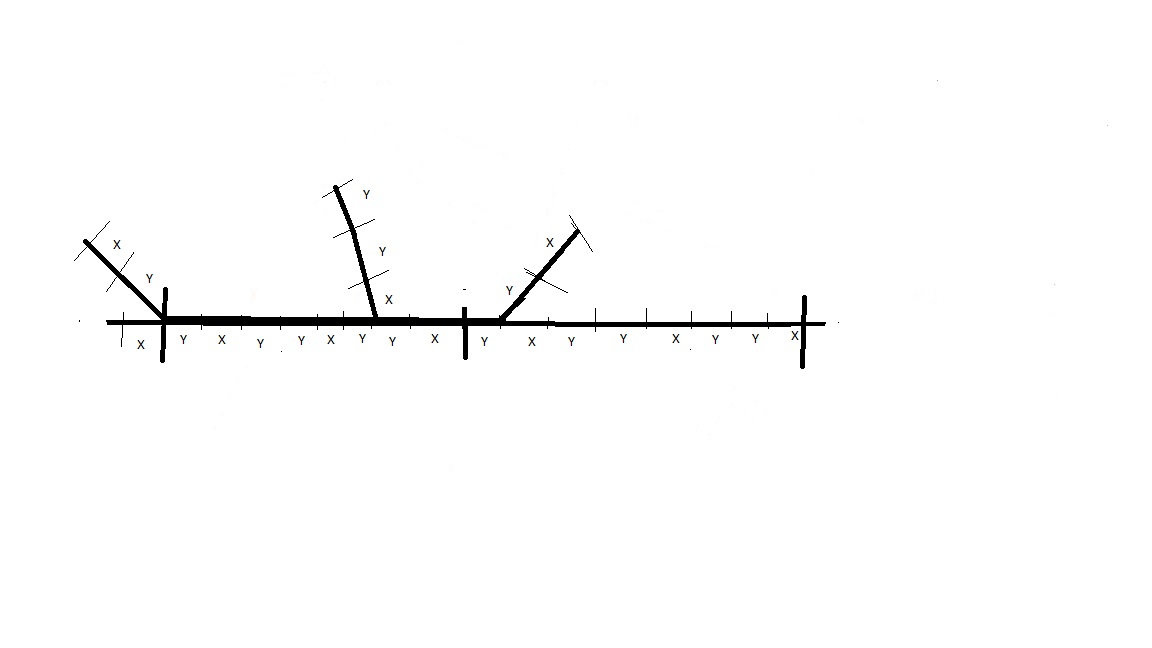}
\captionof{figure}{}

\begin{example} See Figure 6.

Let $W=yxyxyyxyxyyx, W_1=x^{-1}y^{-1}x^{-1}y^{-1}Wyxyx=yyxyxyyxyxyx$, and let $W_2=xWx^{-1}=xyxyxyyxyxyy$.
Then $U=W_1 \cap W=yxyxyyxyxy$ and $V=W_2 \cap \widetilde{W} =yxyxy$, so $U \cap V =$point.

Hence $(W_1 \cup W_2) \cap \widetilde{W} = (U \cup V) \cap \widetilde{W}= Wyxy$.

\end{example}

\includegraphics[scale=0.6]{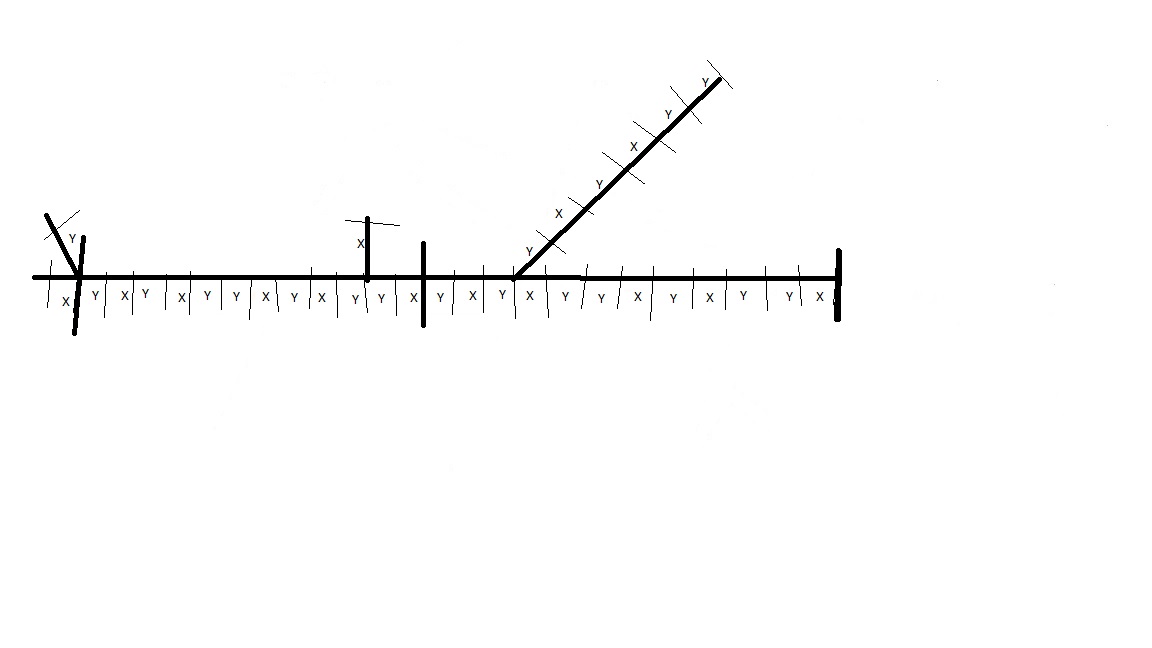}
\captionof{figure}{}

\section{Conjugate Words in a Free Group}

Let $W$ and $\widetilde{W}$ be as in the previous section. Note that $\widetilde{W}$ has a $Z$-shift. 

\begin{lemma}
Assume that there exists an initial subword $U$ of $W$  such that $\widetilde{W}$ contains a nonequivalent (i.e. not obtained by the $Z$-shift) copy of $U$. Call it $U_2$. If $U$ and $U_2$ overlap in such a way that the beginning of $U_2$ lies in $U$ and $L(U \cap U_2) >0$, then there exist decompositions  $U= BC^k$ and $U \cup U_2 = BC^{k+1}$ with $k >0$ such that $B$ is a terminal subword of $C$. If $U \cup U_2$  contains $W$, then there exists a decomposition $W=BC^kI, k > 0$, where $B$ is a terminal subword of $C$ and $I$ is an initial subword of $C$. If $U \cup U_2 =W$, then there exists a decomposition $W=BC^k, k>1$, where $B$ is a terminal subword of $C$.
\end{lemma}

\includegraphics[scale=0.6]{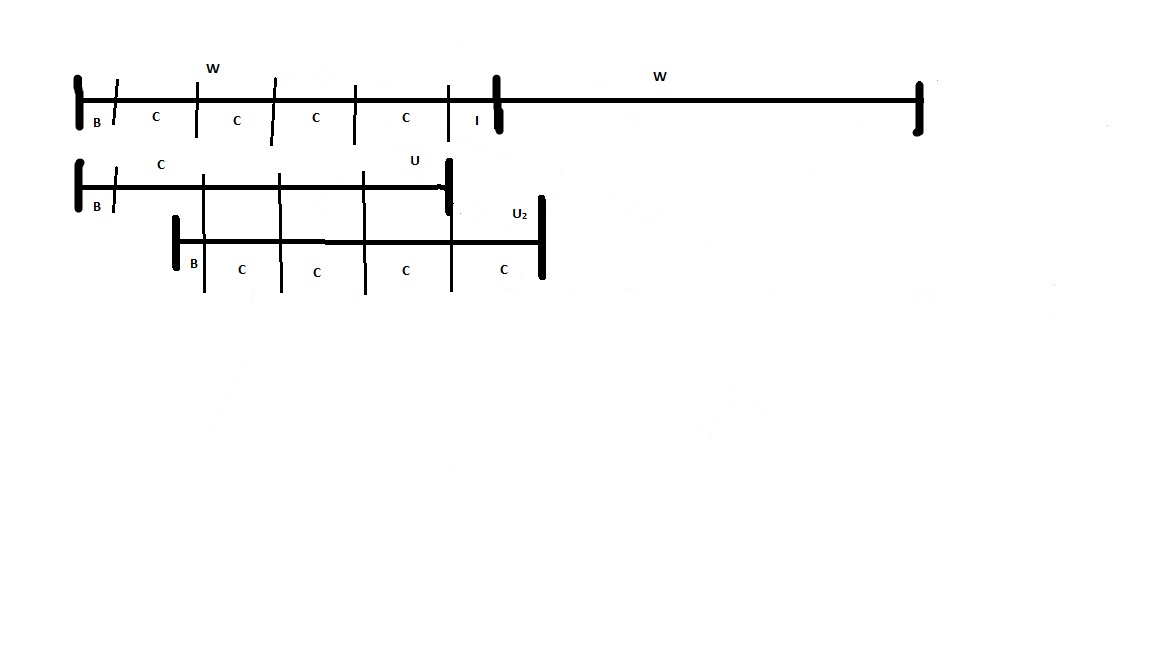}
\captionof{figure}{}

\begin{proof}
Let $P$ be the overlap of $U$ and $U_2$. Then there exists a decomposition $U_2=PC$, and 
$L(U_2)= k \cdot L(C) + n$ with $k > 0$. As $U=U_2$, it follows that $U=BC^k$, where $B$ is a terminal subword of $C$, and $L(B)=n$, see Figure 7. Then $U \cup U_2=BC^{k+1}$. If $U \cup U_2$ contains $W$, then $W$ is a proper initial subword of $BC^{k+1}$ containing $U$. Hence $W=BC^kI, k > 0$, where $B$ is a terminal subword of $C$ and $I$ is an initial subword of $C$. If $U \cup U_2 =W$, then $I$ is trivial and there exists a decomposition $W=BC^k, k>1$, where $B$ is a terminal subword of $C$.
\end{proof}

\begin{remark}
Note that $U \cup U_2$ might be a proper subword of $W$. In that case we do not have much information about $W$.
\end{remark}

\textbf{Proof of Theorem 2}

Let $W_1, W_2, U$ and $V$ be as in the previous section.
Assume that $U \cap V=$point and $U \cup V=W$, hence $U$ is a proper initial subword of $W$ and $V$ is a proper terminal subword of $W$.
Assume WLOG that $L(U)>L(V)$, (the case $L(U)=L(V)$ is considered separately at the end of the section). Then $L(U) > \frac{1}{2} L(W)$.

As the axes are generated by conjugate elements, there exist non-equivalent (i.e. not obtained by the $Z$-shift on $\widetilde{W}$) copies of the words 
$U$ and $V$ in $\widetilde{W}$. As  $L(U) > \frac{1}{2} L(W)$ there exists a non-equivalent copy of $U$ in $\widetilde{W}$ whose beginning is contained in $U$. Call that copy $U_2$.  Also there exists a non-equivalent copy of $V$ in $\widetilde{W}$ whose beginning is contained in $W$. Call that copy $V_2$.

As $U$ and $U_2$ satisfy the conditions of Lemma 2, there exist decompositions  $U= BC^k$ and 
$U \cup U_2 = BC^{k+1}$ with $k >0$ such that $B$ is a terminal subword of $C$.

If $W=U \cup V \subseteq U \cup U_2$, then Theorem 2 follows from Lemma 2.
Hence we need to rule out the case  $U \cup U_2 \subset U \cup V$. 

\bigskip

Assume that $U \cup U_2 \subset U \cup V$. It follows that $V$ and, hence $V_2$, begin with $C$. 

Note that either $V_2 \cap V \neq \varnothing$ or $V_2 \subset U$. 

\bigskip 

Consider $4$ cases.

\textbf{Case 1.}

$V_2 \subset U$ and the beginning $C$ of $V_2$ is "standard" in $U$, i.e. it is one of the $k$ copies of $C$ defined by the decomposition $U = BC^k$.  See Figure 8.

\includegraphics[scale=0.6]{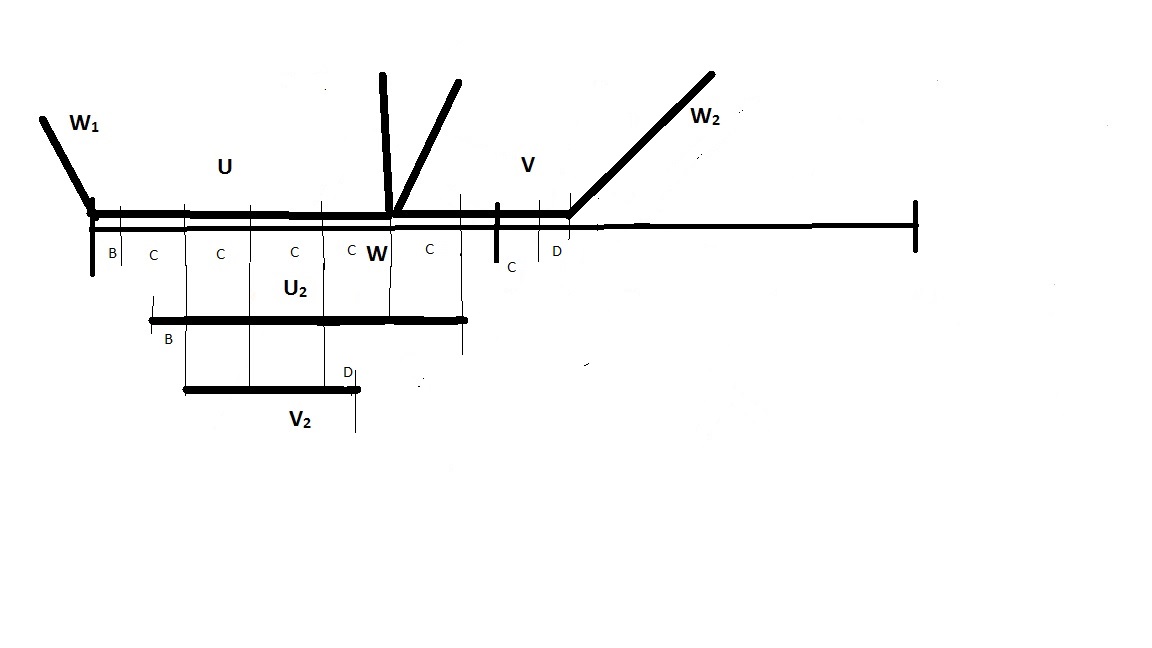}
\captionof{figure}{}

\bigskip

It follows that $V_2=V=C^lD$, where $l>0$ and $D$ is an initial subword of $C$. So the word $U_2$ in $\widetilde{W}$ is followed by the word $C^{l-1}D$. Note that the word $U_2$ in $\widetilde{W}$ corresponds to the word $U$ in $W_1$, so in the word $W_1$  the word $U$ is followed by a copy of the word $C^{l-1}D$, call it $V'$. However, the word $U$ in $W$ is followed by a copy of the word $C$. If
the word $V'$ is non-trivial, it should have non-trivial intersection with that copy of the word $C$ in $W$, so the intersection of $W_1$ with $W$ should be longer than $U$. This contradiction 
implies that $l=1$ and $D$ is trivial, hence $V_2=V=C$. Note that the word $V$ in $\widetilde{W}$ corresponds to the word $V_2$ in $W_2$. As the word $V_2$ in $W$ is preceded by the word $B$, the word $V$ in $W_2$ is preceded by a copy of the word $B$, hence the intersection of $W_2$ and 
$\widetilde{W}$ should be longer than $V$. This contradictions shows that Case 1 cannot happen.

\bigskip
 
\textbf{Case 2.}  
 
$V_2 \subset U$ and the beginning $C$ of $V_2$ is "non-standard" in $U$. See Figure 9.

\includegraphics[scale=0.6]{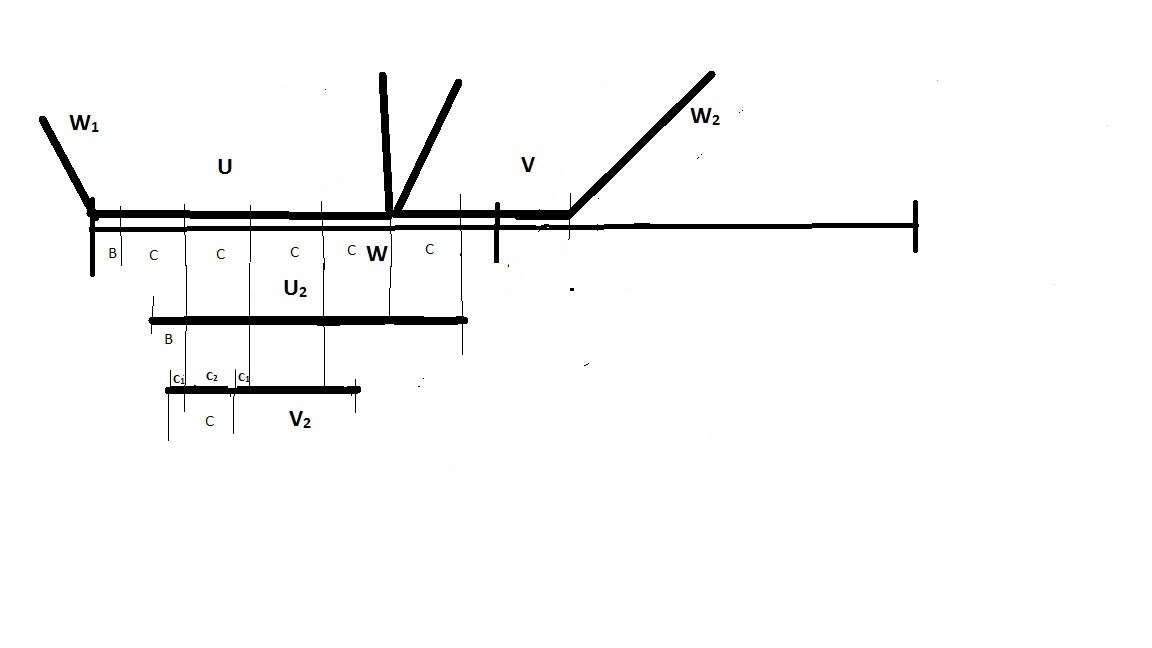}
\captionof{figure}{}

\bigskip

Note that there exist decompositions $C=C_1C_2=C_2C_1$, see Figure 9 and Figure 10.

\includegraphics[scale=0.6]{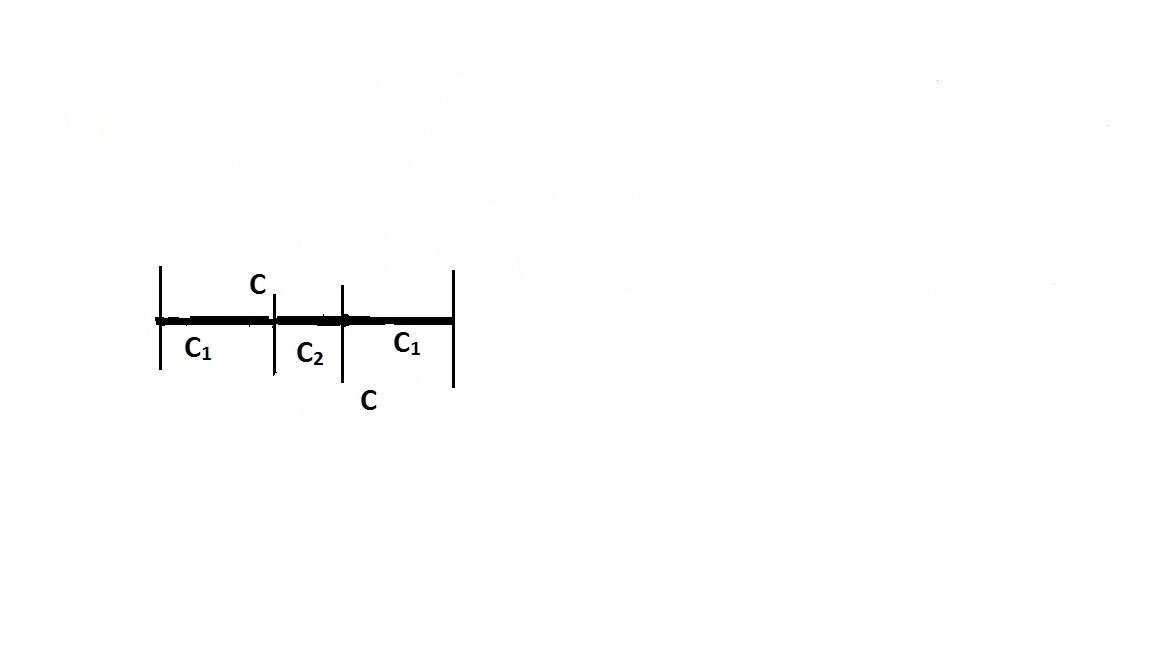}
\captionof{figure}{}

\bigskip

As $C_1$ and $C_2$ commute, they are powers of some $C_0$, \cite{L-S} p.10. It follows that $C=C_0^m, m>1$ and $U=BC^k=BC_0^{km}$. As $B$ is a terminal subword of $C$, it follows that $B=B_0C_0^l, l \ge 0$, where $B_0$ is a terminal subword of $C_0$ which might be empty. So $U=BC^k=B_0C_0^{mk+l}$ and $V_2=V=C_0^nD_0$, where $l>0, n >0$ and $D_0$ is an initial subword of $C_0$. 
Note that $C_0$ is standard in both $U$ and $V$. Hence we can apply the same argument as in Case 1, demonstrating that Case 2 cannot happen.

\bigskip

\textbf{Case 3.}

$V \cap V_2 \neq \varnothing$ and the initial $C$ of $V_2$ is in $U$. 

If the beginning $C$ of $V_2$ is "standard" in $U$, we can use the same argument as in Case 1, to obtain a contradiction.

If the beginning $C$ of $V_2$ is "non-standard" in $U$ , we can use the same argument as in Case 2, to obtain a contradiction. 

Therefore Case 3 is impossible.

\bigskip

\textbf{Case 4.}

$V \cap V_2 \neq \varnothing$ and the initial $C$ of $V_2$ intersects $V$.

We can use the same argument as in Case 2 to obtain a contradiction, so Case 4 is also impossible.

\bigskip

Therefore $U \cup V \subset U \cup U_2$, proving Theorem 2 when $L(U)>L(V)$.

\bigskip

Now consider the case when $L(U)=L(V)=\frac{1}{2}L(W)$.

If $U_2 =V$, then $W=U^2$ and the axes $\lambda, \lambda_1$,and $\lambda_2$ (defined in the previous section) coincide, contradicting their choice to be distinct. 

If $U_2 \neq V$ then $U$ and $U_2$ have a non-trivial intersection. If the beginning of $U_2$ is contained in $U$ then Lemma 2 implies that there exists a decomposition $U \cup U_2 =BC^{k+1}$, where $B,C$, and $k$ are defined above.
If the beginning of $U$ is contained in $U_2$ we can reduce this case to the previous one by considering the words $W_0=U^{-1}V^{-1}, U^{-1}$, and $V^{-1}$ instead of $W, U$, and $V$.

 So $U \cup U_2 \subset W \subseteq U \cup V$. It follows again that $V$ and, hence $V_2$, begin with $C$.
 
Consider the word $V_2$. If $V_2=U$, then $W=V^2$ and the axes $\lambda, \lambda_1$,and $\lambda_2$ (defined in the previous section) coincide, contradicting their choice to be distinct. 

 Otherwise, $V \cap V_2 \neq \varnothing$.   Consider two cases.
 
\bigskip 
 
\textbf{Case 5.} 
 
$V \cap V_2 \neq \varnothing$ and the initial $C$ of $V_2$ is in $U$.

If the beginning $C$ of $V_2$ is "standard" in $U$ then, as in Case 1(above), $V=V_2=C$. As $L(V)=L(U)$, it follows that $U=C$, hence $W=C^2$ and the 
axes  $\lambda, \lambda_1$,and $\lambda_2$ (defined in the previous section) coincide, contradicting their choice to be distinct. 

If the beginning $C$ of $V_2$ is "non-standard" in $U$ then, as in Case 2(above), $C=C_0^m$. Hence as in Case 2(above), it follows that $V=V_2=C_0^m=C$. We got a contradiction with the assumptions that $V_2 \subset U$ and $V \cap V_2 \neq \varnothing$. Therefore this case is impossible.

\bigskip

\textbf{Case 6.} 

$V \cap V_2 \neq \varnothing$ and the initial $C$ of $V_2$ intersects $V$. 

As in Case 2(above) it follows that $V=V_2=C$, hence  as $L(V)=L(U)$ it follows that $W=U^2=C^2$ and the axes $\lambda, \lambda_1$,and $\lambda_2$ (defined in the previous section) coincide, contradicting their choice to be distinct. 

Therefore in the special case when $L(U)=L(V)=\frac{1}{2}L(W)$ the axes $\lambda, \lambda_1$,and $\lambda_2$ (defined in the previous section) coincide, contradicting their choice to be distinct. 
That contradiction completes the proof of Theorem 2.

\begin{remark}  
Note that there exists a decomposition $C=IT$, where $T$ is a terminal subword of $C$. If $T=B$, then $W$ is a conjugate of $C^{k+1}$, so the axes
 $\lambda, \lambda_1$,and $\lambda_2$ (defined in the previous section) coincide, contradicting their choice to be distinct. 
What can be said about $W$ if $T \neq B$? 
\end{remark}

The following conjecture was formulated by Max Neumann-Coto.

\textbf{Conjecture}
\emph{Assume that $W$ overlaps two of its conjugates in such a way that the overlaps cover all of $W$ and the overlaps do not intersect. Then $W=DC^k$, where $C$ is non-trivial and $k >1$ and the conjugates have the form $C^rDC^{k-r}$ and $C^sDC^{k-s}$.}

\section{Acknowledgment}
The author would like to thank Peter Scott for helpful conversations.

\end{document}